\documentclass[11pt, reqno]{amsart}

\usepackage{amsmath,amssymb,amsthm, epsfig}
\usepackage{hyperref}
\usepackage{amsmath}
\usepackage{amssymb}
\usepackage{color}
\usepackage{ulem}
\usepackage{epsfig}
\usepackage[mathscr]{eucal}
\usepackage[latin1]{inputenc}

\usepackage{hyperref}
\hypersetup{colorlinks=true, linkcolor=blue, citecolor=blue, linktoc=page, pdftitle={Homogenization in Orlicz-Sobolev}, pdfauthor={Marcon, Rodrigues, Teymurazyan}}

\usepackage{enumerate}

\newtheorem{theorem}{Theorem}
\newtheorem{definition}{Definition}

\newtheorem{lemma}{Lemma}

\newtheorem{remark}{Remark}

\date{}
\numberwithin{equation}{section}
\numberwithin{theorem}{section}
\numberwithin{lemma}{section}
\numberwithin{corollary}{section}
\numberwithin{remark}{section} \numberwithin{proposition}{section}
\numberwithin{definition}{section}

\def \Div {\mathrm{div}}

\def \R {\mathbb{R}}

\def \per{\mathrm{per}}

\begin{document}

\title[Homogenization of obstacle problems]{Homogenization of obstacle problems in Orlicz-Sobolev spaces}

\author[D. Marcon]{Diego Marcon}
\address{Department of Mathematics, Federal University of Rio Grande do Sul, Av. Bento Gon\,{c}alves 9500, 91509-900 Porto Alegre, RS, Brazil.}
\email{diego.marcon@ufrgs.br}

\author[J.F. Rodrigues]{Jos\'{e} Francisco Rodrigues}
\address{CMAFcIO, Department of Mathematics, Faculty of Sciences, University of Lisbon, Portugal.}
\email{jfrodrigues@ciencias.ulisboa.pt}

\author[R. Teymurazyan]{Rafayel Teymurazyan}
\address{CMUC, Department of Mathematics, University of Coimbra, 3001-501 Coimbra, Portugal.}
\email{rafayel@utexas.edu}

\begin{abstract}
We study the homogenization of obstacle problems in Orlicz-Sobolev spaces for a wide class of monotone operators (possibly degenerate or singular) of the $p(\cdot)$-Laplacian type. Our approach is based on the Lewy-Stampacchia inequalities, which then give access to a compactness argument. We also prove the convergence of the coincidence sets under non-degeneracy conditions.

\bigskip


\bigskip

\noindent \textbf{AMS Subject Classifications MSC 2010:} 35J20, 35J62, 35J92, 35D30, 35A15.

\end{abstract}

\maketitle

\section{Introduction}

Let $\Omega\subset\R^n$ be a bounded domain and $p:\Omega\rightarrow\R$ be measurable and such that 
\begin{equation}\label{1.1}
1<\alpha\leq p(x)\leq\beta<\infty\,\,\,\textrm{ a.e. in } \Omega,
\end{equation}
where $\alpha$ and $\beta$ are constants. The following variable exponent Lebesgue space is an Orlicz space:
$$
L^{p(\cdot)}(\Omega):=\left\{u:\Omega\rightarrow\R\,\textrm{ measurable }\,\rho(u):=\int_\Omega\left|u(x)\right|^{p(x)}\,dx<\infty\right\}.
$$
This Orlicz space is a separable reflexive Banach space with the following (Luxembourg) norm:
$$
\|u\|_{L^{p(\cdot)}(\Omega)}:=\inf\left\{\lambda>0,\,\,\rho\left(\frac{|u|}{\lambda}\right)\leq1\right\}.
$$
We define an Orlicz-Sobolev space by
$$
W^{1,p(\cdot)}(\Omega):=\left\{u\in L^{p(\cdot)}(\Omega),\,\,\nabla u\in\left(L^{p(\cdot)}(\Omega)\right)^n\right\},
$$
with the norm
$$
\|u\|_{W^{1,p(\cdot)}(\Omega)}:=\|u\|_{L^{p(\cdot)}(\Omega)}+\|\nabla u\|_{L^{p(\cdot)}(\Omega)},\,\,\,\|\nabla u\|_{L^{p(\cdot)}(\Omega)}=\sum_{i=1}^n\left\|\frac{\partial u}{\partial x_i}\right\|_{L^{p(\cdot)}(\Omega)}.
$$
This Orlicz-Sobolev space is also a separable and reflexive Banach space. We also define
$$
W_0^{1,p(\cdot)}(\Omega):=\left\{u\in W_0^{1,1}(\Omega),\,\,\rho(|\nabla u|)<\infty\right\}.
$$
The latter is a Banach space endowed with the norm
$$
\|u\|_{W^{1,p(\cdot)}_0(\Omega)}:=\|\nabla u\|_{L^{p(\cdot)}(\Omega)}.
$$
In this paper we study the periodic homogenization of obstacle problems in Orlicz-Sobolev spaces. We consider
$$
a(x,\xi):\Omega\times\R^n\rightarrow\R^n
$$
to be a Carath\'{e}odory vector function, that is, we assume it is continuous with respect to $\xi$, for almost every $x\in\R^n$, and that it is measurable with respect to $x$, for every $\xi$. Moreover, the functions $a(\cdot,\xi)$ and $p(\cdot)$ are assumed to be periodic with period $1$ in each argument $x_1$, $x_2$, $\ldots$, $x_n$. We denote the periodicity cell by $Q$, i.e. $Q:=(0,1]^n$. Additionally, we assume that the following structural conditions (monotonicity, coercitivity and boundedness) hold:
\begin{equation}\label{1.3}
\begin{cases}
&\left(a(x,\xi)-a(x,\eta)\right)\cdot(\xi-\eta)>0,\,\,\textrm{ for a.e. }x,\,\,\xi\neq\eta,\\

&a(x,\xi)\cdot\xi\geq C_1\left(|\xi|^{p(x)}-1\right),\\
&\left|a(x,\xi)\right|\leq C_2\left(|\xi|^{p(x)-1}+1\right),
\end{cases}
\end{equation}
where $C_1$, $C_2>0$ are constants. For $\varepsilon>0$, we define
\begin{equation}\label{1.4}
a_\varepsilon(x,\xi):=a\left(\frac{x}{\varepsilon},\xi\right),\,\,\,x\in\Omega,\,\,\xi\in\R^n
\end{equation}
and $p_\varepsilon(x)=p(x/\varepsilon)$. The Orlicz-Sobolev spaces of periodic functions, denoted by $W^{1, p(\cdot)}_{\per}(Q)$, is defined as the set of periodic functions $u$ from $W^{1,1}_{\per}(Q)$ with
$$
\int_Q u\,dx=0\qquad\textrm{ and }\qquad\int_Q\left|\nabla u\right|^{p(x)}\,dx<\infty.
$$
For the homogenized functional defined by
\begin{equation}\label{1.30}
h(\xi):=\min_{v\in W^{1,p(\cdot)}_{\per}(Q)}\int_Q\frac{\left|\xi+\nabla v\right|^{p(x)}}{p(x)}\,dx, 
\end{equation}
we introduce also the Orlicz-Sobolev spaces
$$
W^h(\Omega):=\left\{u\in W^{1,1}(\Omega),\,\,h(\nabla u)\in L^1(\Omega)\right\},$$
$$
W_0^h(\Omega):=\left\{u\in W^{1,1}_0(\Omega),\,\,h(\nabla u)\in L^1(\Omega)\right\},\,
$$
with the norm, $\|u\|_{W_0^h(\Omega)}:=\|\nabla u\|_{L^h(\Omega)}$, and the vector Orlicz space
$$
L^h(\Omega):=\left\{\xi\in \left[L^1(\Omega)\right]^n,\,\,h(\xi)\in L^1(\Omega)\right\},
$$
normed by
$$
\|\xi\|_{L^h(\Omega)}:=\inf\left\{\lambda>0,\,\,\int_\Omega h\left(\frac{\xi}{\lambda}\right)\leq1\right\}.
$$
By the properties of $h$, as it was observed in \cite{ZP102}, we have the continuous embeddings
$$
L^{\beta}(\Omega)\subset L^h(\Omega)\subset L^{\alpha}(\Omega),
$$

Assuming that
\begin{equation}\label{1.40}
f\,\,\hbox{ and }\,\,\left(A_\varepsilon\psi_\varepsilon-f\right)^+\in L^s(\Omega),
\end{equation}
\begin{equation}\label{1.41}
\left\|\left(A_\varepsilon\psi_\varepsilon-f\right)^+\right\|_{L^s(\Omega)}\leq C,
\end{equation}
where $C>0$ is a constant independent of $\varepsilon$ and
\begin{equation}\label{1.42}
\psi_\varepsilon\in W^{1,p_\varepsilon(\cdot)}(\Omega),\,\,\psi_0\in W^{h}(\Omega),\,\,\psi_\varepsilon^+\in W_0^{1,p_\varepsilon(\cdot)}(\Omega),\,\,\psi_0^+\in W_0^{h}(\Omega),
\end{equation}
where $\alpha'=\alpha/(\alpha-1)$, $u^+$ is the positive part of $u$ and $s>\frac{n\alpha'}{n+\alpha'}$ if $\alpha<n$, $s>1$, if $\alpha=n$ and $s=1$ for $\alpha>n$, we show (Theorem \ref{t3.1}) that the unique solution $u_\varepsilon\in K_\varepsilon$ of the obstacle problem
\begin{equation}\label{1.5}
\int_\Omega a_\varepsilon(x,\nabla u_\varepsilon)\cdot\nabla(v-u_\varepsilon)\,dx\geq\int_\Omega f(v-u_\varepsilon)\,dx,\,\,\,\forall v\in K_\varepsilon,
\end{equation}
where 
$$
K_\varepsilon:=\left\{v\in W_0^{1,p_\varepsilon(\cdot)}(\Omega),\,v\geq\psi_\varepsilon\,\textrm{ a.e. in }\,\Omega\right\},
$$
converges to the unique solution $u_0\in K_0$ of the following homogenized obstacle problem
\begin{equation}\label{1.6}
\int_\Omega a_0(\nabla u_0)\cdot\nabla(v-u_0)\,dx\geq\int_\Omega f(v-u_0)\,dx,\,\,\,\forall v\in K_0,
\end{equation}
where
$$
K_0:=\left\{v\in W_0^{h}(\Omega),\,v\geq\psi_0\,\textrm{ a.e. in }\,\Omega\right\}.
$$
The homogenized operator $a_0:\R^n\rightarrow\R^n$ is given in terms of the weighted average of $a$ as in \cite{ZP102}, that is,
\begin{equation}\label{2.1}
a_0(\xi):=\int_Qa(x,\xi+\nabla v)\,dx,
\end{equation}
with $v\in W_{\per}^{1,p(\cdot)}(Q)$, such that,
$$
\int_Qa(x,\xi+\nabla v)\cdot\nabla\varphi\,dx=0,\,\,\,\forall\varphi\in W_{\per}^{1,p(\cdot)}(Q),
$$
where $Q$ is the periodicity cell. 

Note that, due to the Lavrent'ev effect, if instead of $W_{\per}^{1,p(\cdot)}(Q)$, we take $\varphi\in C_{\per}^\infty(Q)$, we may end up with a different homogenized operator, since in general the space $C_{\per}^\infty(Q)$ is not dense in $W_{\per}^{1,p(\cdot)}(\Omega)$. These homogenized operators, referred to as $W$ and $H$ solutions in \cite{ZP102}, respectively, in general may be different, but our results hold for both solutions, with minor modifications for the space framework of the $H$ solutions. Although we prefer to work with $W$ solutions, that is due to the fact that \cite[Theorem 3.1]{ZP102} (see Theorem \ref{t2.1} below) is true for both types of solutions. Observe that we do not impose any regularity assumption on $p(\cdot)$. However, in the particular case when $p$ is log-Lipschitz continuous, i.e., when for a constant $L>0$
\begin{equation*}\label{1.10}
-|p(x)-p(y)|\log|x-y|\leq L,\,\,\,\forall x,y\in\overline{\Omega}, \,\,|x-y|<1/2,
\end{equation*}
the notion of $W$ and $H$ solutions coincide (see \cite{DHHR11,FZ01}), since then the smooth functions are dense in the Orlicz-Sobolev space.

Our approach is a development of the classical methods \cite{CD99,E90} (see also \cite{S05, S06, ZP102}) combined with the Lewy-Stampacchia inequalities in the Orlicz-Sobolev framework, in accordance with \cite{RT11}, which then allows the use of a Rellich-Kondrachov compactness argument.

The result generalizes, in part, that of \cite{BM91}, which covers the case when $p$ is constant (and hence the homogenization is in usual Sobolev spaces). The latter, in turn, implies the case of $p=2$ obtained in \cite{BM82}. Nonetheless, we observe that the structural assumptions \eqref{1.3} allow us to consider a wider range of monotone operators, which cover these cases and include other interesting quasilinear operators, some of which we list below.

\begin{enumerate}[$1.$]
\item If $a(x,\xi)=|\xi|^{p(x)-2}\xi$, we deal with the obstacle problem for the $p(x)$-Laplace operator.

\item We can also consider perturbations of the $p$-Laplace ($p$ constant) and of the $p(x)$-Laplace operators, taking 
$$
a(x,\xi)=\gamma(x)|\xi|^{p-2}\xi\,\,\textrm{ and }\,\,a(x,\xi)=\gamma(x)|\xi|^{p(x)-2}\xi
$$
for any non-negative bounded periodic function $\gamma(x)$.

\item It is possible to consider functions which are essentially different from these previous ``power like'' functions. One general example can be
$$
a(x,\xi)= \gamma_1(x)|\xi|^{p(x)-1}\xi\log\left(\gamma_2(x)|\xi|+\gamma_3(x)\right),
$$
where $\gamma_3(x)$, $p(x)>1$ and $\gamma_1(x)$, $\gamma_2(x)>0$ a.e. in $\Omega$ are bounded periodic functions.
\end{enumerate}


The paper is organized as follows: in Section \ref{s2}, we state some preliminaries facts, which then serve to prove our main result in Section \ref{s3} (Theorem \ref{t3.1}). In Section \ref{s4}, we prove the convergence of the coincidence sets (Theorems \ref{t4.1} and \ref{t4.2}).
\section{Preliminaries}\label{s2}
In this section we give some preliminaries. In particular, we provide the concept of $G$-convergence of operators in our framework, as well as convergence of sets in Mosco sense. We also recall some results from  \cite{Z91} and \cite{ZP102} for future reference. We start by setting some notations, which will be used throughout the paper: $p_\varepsilon(x)=p(x/\varepsilon)$; $ \alpha'=\frac{\alpha}{\alpha-1}$; $\rightharpoonup$ denotes the weak convergence;
$$
\displaystyle A_\varepsilon u:=-\Div\big(a_\varepsilon(x,\nabla u)\big)\,\,\hbox{ and }\,\,\displaystyle A_0u:=-\Div\big(a_0(\nabla u)\big),
$$
where $a_\varepsilon$ is defined by \eqref{1.4}, and $a_0$ is defined by \eqref{2.1}. Next, we define the notion of $G$-convergence of $a_\varepsilon$ to $a_0$. Observe, that most definitions of $G$-convergence that can be found in the literature (see, for example, \cite{B11,BD75,CDD90,P97}), allow $a_0$ to depend on $x$ as well, just as $a_{\varepsilon}$ depends. However, in some particular cases, more information can be said about the limiting operator. One example is that of operators with rapidly oscillating ``coefficients''. Since our assumptions ensure that $a(x,\xi)$ and $p(\cdot)$ are periodic with respect to $x$ in each of the arguments $x_1$, $x_2$, $\ldots$, $x_n$, there is no loss in generality to impose $a_0$ to be independent of $x$ in the definition of $G$-convergence, which is more relevant for our purposes.
\begin{definition}\label{d2.1}
Consider $a_{\varepsilon}:\Omega\times\R^n\rightarrow\R^n$ and $a_0:\R^n\rightarrow\R^n$ as above. We say that $a_\varepsilon$ $G$-converges to $a_0$ when, considering the unique solution $u_\varepsilon \in W_0^{1,p_\varepsilon(\cdot)}(\Omega)$ of 
$$
-\Div\left(a_\varepsilon(x,\nabla u_\varepsilon)\right)=f,\,\,f\in W^{-1,\alpha'}_0(\Omega)\textrm{ in }\,\mathcal{D}'(\Omega)
$$
and $u_0\in W_0^h(\Omega)$ the unique solution of 
$$
-\Div\left(a_0(\nabla u_0)\right)=f\,\,\textrm{ in }\,\mathcal{D}'(\Omega),
$$
there holds:
\begin{enumerate}[$(1)$]
	\item $u_\varepsilon\rightharpoonup u_0$ in $W_0^{1,\alpha}(\Omega)$, as $\varepsilon\rightarrow0$;
	\item $\displaystyle a_\varepsilon(x,\nabla u_\varepsilon)\rightharpoonup a_0(\nabla u_0)$ in $\left(L^{\beta'}(\Omega)\right)^n$, as $\varepsilon\rightarrow0$.
\end{enumerate}
\end{definition}
Note that the choice of $s$ in \eqref{1.40} guarantees, in particular, $f\in W^{-1,\alpha'}(\Omega)$. Additionally, $a(x,\xi)$ is assumed to be continuous with respect to $\xi$, for almost every $x\in\R^n$.

Next, we state a theorem from \cite[Theorem 3.1]{ZP102} that insures the $G$-convergence of $a_\varepsilon$ to a function $a_0$, as $\varepsilon\rightarrow0$, given explicitly in terms of $a$. Its proof is based on a compensated compactness argument from \cite{ZP101,ZP102}, which, in the case of $p(\cdot)=$ constant, resembles the well known result of Tartar-Murat (see \cite{M78}).
\begin{theorem}\label{t2.1}
Let $a(x,\xi)$ be a Carath\'{e}odory vector function, which is periodic with respect to $x$ in each argument and satisfy \eqref{1.3}. Let also $p$ be periodic, measurable and satisfy \eqref{1.1}. If structural conditions \eqref{1.3} hold, then $a_\varepsilon$ $G$-converges to $a_0$, where $a_0$ is defined by \eqref{2.1}. Moreover,
$$
\int_\Omega a_\varepsilon(x,\nabla u_\varepsilon)\cdot\nabla u_\varepsilon\,dx\rightarrow\int_\Omega a_0(\nabla u_0)\cdot\nabla u_0\,dx,
$$
as $\varepsilon\rightarrow0$.
\end{theorem}
As it is shown in \cite{ZP102}, the vector function $a_0(\xi)$ is strictly monotone, i.e.,
$$
\left(a_0(\xi)-a_0(\eta)\right)\cdot(\xi-\eta)>0,\,\,\,\xi\neq\eta,
$$
and coercive, that is, 
$$
a_0(\xi)\cdot\xi>c_0(h(\xi)-1),
$$
where $c_0>0$ is a constant, and the homogenized functional $h(\xi)$ is defined by \eqref{1.30}. Moreover, $h$ satisfies the so-called $\Delta_2$ condition, \cite[Proposition 2.1]{ZP102}, which implies that the Orlicz space $L^h(\Omega)$ is reflexive. As it is observed in \cite{ZP102}, $h(\xi)$ being defined by \eqref{1.30}, is convex on $\R^n$ and satisfies the following two-sided estimate:
$$
c_1|\xi|^\alpha-1\leq h(\xi)\leq c_2|\xi|^\beta+1,
$$
for a $c_1>0$ constant. As a consequence, we have 
$$
W_0^{1,\beta}(\Omega)\subset W_0^h(\Omega)\subset W_0^{1,\alpha}(\Omega),
$$
which implies that 
$$
K_0\subset W_0^{1,\alpha}(\Omega).
$$
The following result is from \cite{Z91}, and it provides more information on the homogenized functional.
\begin{lemma}\label{l2.1}
If $u_\varepsilon$ is a sequence uniformly bounded in $W_0^{1,p_\varepsilon(\cdot)}(\Omega)$, such that, $u_\varepsilon\rightharpoonup u_0$ in $W_0^{1,\alpha}(\Omega)$ as $\varepsilon\rightarrow0$, then $h(\nabla u_0)\in L^1(\Omega)$.
\end{lemma}
Observe that Lemma \ref{l2.1} guarantees that, within $G$-convergence, the weak limits of $u_\varepsilon$ in $W_0^{1,\alpha}(\Omega)$ belong to $W_0^h(\Omega)$, and therefore, if also $u_\varepsilon\in K_\varepsilon$ then $u_0\in K_0$.

In order to state our main result, we will also need to redefine the Mosco convergence of sets.
\begin{definition}\label{d2.2}
The sequence of closed convex sets $K_\varepsilon\subset W_0^{1,p_\varepsilon(\cdot)}(\Omega)$, is said to converge to the set $K_0\subset W_0^{h}(\Omega)$ in the Mosco sense, if
\begin{itemize}
\item for any $v_0\in K_0$ there exists a sequence $v_\varepsilon\in K_\varepsilon$, such that, $v_\varepsilon\rightarrow v_0$ in $W_0^{1,\alpha}(\Omega)$;
\item weak limits in $W_0^{1,\alpha}(\Omega)$ of any sequence of elements in $K_\varepsilon$, that is uniformly bounded in $W_0^{1,p_\varepsilon(\cdot)}(\Omega)$, belong to $K_0$.
\end{itemize}
\end{definition}
\begin{remark}\label{r2.1}
Since $W_0^{1,p_\varepsilon(\cdot)}(\Omega)$ is continuously embedded into $W_0^{1,\alpha}(\Omega)$ (see, for example, \cite{DHHR11}), then $\psi_\varepsilon\rightarrow\psi_0$ in $W^{1,\beta}(\Omega)$ provides $K_\varepsilon\rightarrow K_0$ in the Mosco sense, where $K_\varepsilon$ and $K_0$ are as in \eqref{1.5} and \eqref{1.6} respectively. \end{remark}
\smallskip
\section{Homogenization of the obstacle problem}\label{s3}
We are now ready to prove our main result, which states as follows. 
\begin{theorem}\label{t3.1}
Let $a(x,\xi)$ be a Carath\'{e}odory vector function satisfying \eqref{1.3} and periodic with respect to $x$ in each argument. Let $p(\cdot)$ be periodic, measurable and satisfying \eqref{1.1}. Assume further that \eqref{1.40}-\eqref{1.42} hold. If $K_\varepsilon\rightarrow K_0$ in the Mosco sense, then the unique solution of \eqref{1.5} converges weakly in $W_0^{1,\alpha}(\Omega)$, as $\varepsilon\rightarrow0$, to the unique solution of \eqref{1.6}, where $a_0$ is given by \eqref{2.1}.
\end{theorem}
\begin{proof}
We divide the proof into five steps.
\smallskip

\noindent\textbf{Step 1} (\textit{Apriori estimates}). Existence and uniqueness of the solution of \eqref{1.5} (and \eqref{1.6}) is a classical result (see, for instance, \cite{CLRT14,R91,RT11}). As in the proof of \cite[Theorem 2.3]{BM91} (see also \cite[page 145]{R91}), the coercitivity and boundedness assumptions from \eqref{1.3} 
imply that $u_\varepsilon$ is bounded in $W_0^{1,p_\varepsilon(\cdot)}(\Omega)$ by a constant depending only from $C_1$, $C_2$ but independent of $\varepsilon$. For the details we refer the reader to \cite{FZ03}. As a consequence we obtain that $u_\varepsilon$ is bounded also in $W_0^{1,\alpha}(\Omega)$, since $W_0^{1,p_\varepsilon(\cdot)}(\Omega)\subset W_0^{1,\alpha}(\Omega)$. Set
\begin{equation}\label{3.1}
\sigma_\varepsilon:=a_\varepsilon(x,\nabla u_\varepsilon),\qquad \mu_\varepsilon:=-\Div\left(a_\varepsilon(x,\nabla u_\varepsilon)\right)-f.
\end{equation}
The boundedness condition from \eqref{1.3} implies that $\sigma_\varepsilon$ and $\mu_\varepsilon$ are bounded (see \cite{BM91,ZP102}), therefore we can extract weakly convergent subsequence (still denoted by $\varepsilon$) from each one of them. Thus, there exist $u^*$, $\sigma^*$, $\mu^*$ such that
\begin{equation}\label{3.2}
u_\varepsilon\rightharpoonup u^*\,\,\,\textrm{ in }W_0^{1,\alpha}(\Omega)\,\textrm{ and } u_\varepsilon\rightarrow u^*\,\,\textrm{ in }\,\,L^{\alpha}(\Omega),
\end{equation}
\begin{equation}\label{3.3}
\sigma_\varepsilon\rightharpoonup\sigma^*\,\,\,\textrm{ in }\left(L^{\beta'}(\Omega)\right)^n,
\end{equation}
\begin{equation}\label{3.4}
\mu_\varepsilon\rightharpoonup\mu^*\,\,\,\textrm{ in }W^{-1,\beta'}(\Omega).
\end{equation}
Note that
\begin{equation}\label{3.5}
\mu^*=-\Div\sigma^*-f.
\end{equation}
Moreover, using Lemma 2.1 and since $K_\varepsilon\rightarrow K_0$ in the Mosco sense, then\begin{equation}\label{3.6}
u^*\in K_0.
\end{equation}
\noindent\textbf{Step 2} (\textit{Compactness}). Note that our assumptions provide the Lewy-Stampacchia inequalities (see \cite{RT11}), that is, we have
$$
f\leq f+\mu_\varepsilon\leq(A_\varepsilon\psi_\varepsilon-f)^++f,
$$
which implies, by a Rellich-Kondrachov compactness argument,
\begin{equation}\label{3.7}
\mu_\varepsilon\rightarrow\mu^*\,\,\,\textrm{ in }W^{-1,\alpha'}(\Omega).
\end{equation}
\noindent\textbf{Step 3.} In this step we prove that $\sigma^*=a_0(\nabla u^*)$, where $a_0$ is defined by \eqref{2.1}. To see this, let $w_0\in\mathcal{D}(\Omega)$ and $w_\varepsilon\in W_0^{1,p_\varepsilon(\cdot)}(\Omega)$ be the unique solution of
\begin{equation}\label{3.8}
\Div\left(a_\varepsilon(x,\nabla w_\varepsilon)\right)=\Div\left(a_0(\nabla w_0)\right)\,\,\,\textrm{ in }\mathcal{D}'(\Omega).
\end{equation}
From Theorem \ref{t2.1}, we have that $a_\varepsilon$ $G$-converges to $a_0$, as $\varepsilon\rightarrow0$, where $a_0(\xi)$ is defined by \eqref{2.1}. In particular,
\begin{equation}\label{3.9}
\begin{cases}
&w_\varepsilon\rightharpoonup w_0\,\,\,\textrm{ in } W^{1,\alpha}(\Omega)\\
& a_\varepsilon(x,\nabla w_\varepsilon)\rightharpoonup a_0(\nabla w_0)\,\,\,\textrm{ in }\left(L^{\beta'}(\Omega)\right)^n.
\end{cases}
\end{equation}
Fix now $\varphi$ such that
\begin{equation}\label{3.10}
\varphi\in\mathcal{D}(\Omega),\,\,\,0\leq\varphi\leq1.
\end{equation}
From the monotonicity of $a_\varepsilon$ one has
\begin{equation}\label{3.11}
\int_\Omega\varphi\left(a_\varepsilon(x,\nabla u_\varepsilon)-a_\varepsilon(x,\nabla w_\varepsilon)\right)\cdot(\nabla u_\varepsilon-\nabla w_\varepsilon)\,dx\geq0.
\end{equation}
Since $u^*\in K_0$, and $K_\varepsilon\rightarrow K_0$ in the Mosco sense, there exists a sequence $\bar{u}_\varepsilon$, such that,
\begin{equation}\label{3.12}
\bar{u}_\varepsilon\in K_\varepsilon\,\,\,\textrm{ and }\,\,\,\bar{u}_\varepsilon\rightarrow u^*\,\,\,\textrm{ in }W_0^{1,\alpha}(\Omega).
\end{equation}
Next, we write \eqref{3.11} as
\begin{eqnarray}\label{3.13}
  \int_\Omega\varphi\sigma_\varepsilon\cdot(\nabla u_\varepsilon-\nabla\bar{u}_\varepsilon)  &+& \int_\Omega\varphi\sigma_\varepsilon\cdot\nabla\bar{u}_\varepsilon-\int_\Omega\varphi\sigma_\varepsilon\cdot\nabla w_\varepsilon\nonumber\\
  &-& \int_\Omega\varphi a_\varepsilon(x,\nabla w_\varepsilon)\cdot\nabla(u_\varepsilon-w_\varepsilon)\nonumber\\
  &:=&I_1+I_2+I_3+I_4.
\end{eqnarray}
Since $0\leq\varphi\leq1$ on $\Omega$, and $K_\varepsilon$ is convex, then the function $v=\varphi\bar{u}_\varepsilon+(1-\varphi)u_\varepsilon$ can be used as a test function in \eqref{1.5}, which gives
\begin{equation}\label{3.14}
\int_\Omega\sigma_\varepsilon\cdot\nabla\big(\varphi(\bar{u}_\varepsilon-u_\varepsilon)\big)\geq\int_\Omega f\varphi(\bar{u}_\varepsilon-u_\varepsilon)
\end{equation}
and so
\begin{eqnarray*}
I_1 &=& \displaystyle{\int_\Omega}\sigma_\varepsilon\cdot\nabla\big(\varphi(u_\varepsilon-\bar{u}_\varepsilon)\big)-\displaystyle{\int_\Omega}(u_\varepsilon-\bar{u}_\varepsilon)\sigma_\varepsilon\cdot\nabla\varphi\nonumber\\ &\leq&\int_\Omega f\varphi(u_\varepsilon-\bar{u}_\varepsilon)-\displaystyle{\int_\Omega}(u_\varepsilon-\bar{u}_\varepsilon)\sigma_\varepsilon\cdot\nabla\varphi.
\end{eqnarray*}
Since $u_\varepsilon$ and $\bar{u}_\varepsilon$ converge to $u^*$ weakly in $W_0^{1,\alpha}(\Omega)$ (and strongly in $L^{\alpha}(\Omega)$), we obtain
\begin{equation}\label{3.15}
\limsup_{\varepsilon\rightarrow0} I_1\leq0.
\end{equation}
As we know from \eqref{3.12}, $\bar{u}_\varepsilon\rightarrow u^*$ in $W_0^{1,\alpha}(\Omega)$, which gives
\begin{equation}\label{3.16}
\lim_{\varepsilon\rightarrow0}I_2=\int_\Omega\varphi\sigma^*\cdot\nabla u^*.
\end{equation}
Note that
$$
I_3=-\int_\Omega\sigma_\varepsilon\cdot\nabla\left(\varphi w_\varepsilon\right)+\int_\Omega w_\varepsilon\sigma_\varepsilon\cdot\nabla\varphi.
$$
From \eqref{3.7} and \eqref{3.12}, we pass to the limit in the first term of $I_3$. Using \eqref{3.3} and \eqref{3.12}, we pass to the limit also in the second term of $I_3$, arriving at
\begin{equation}\label{3.17}
\lim_{\varepsilon\rightarrow0}I_3=-\int_\Omega\varphi\sigma^*\nabla w_0.
\end{equation}
Observe that
$$
I_4=-\int_\Omega a_\varepsilon(x,\nabla w_\varepsilon)\cdot\nabla\left(\varphi(u_\varepsilon-w_\varepsilon)\right)+\int_\Omega(u_\varepsilon-w_\varepsilon)a_\varepsilon(x,\nabla w_\varepsilon)\cdot\nabla\varphi,
$$
and recalling \eqref{3.2} and \eqref{3.9} and passing to the limit we obtain
\begin{equation}\label{3.18}
\lim_{\varepsilon\rightarrow0}I_4=-\int_\Omega\varphi a_0(\nabla w_0)\cdot\nabla(u^*-w_0).
\end{equation}
Combining \eqref{3.13}, \eqref{3.15}-\eqref{3.18}, one has
\begin{equation}\label{3.19}
\int_\Omega\varphi(\sigma^*-a_0(\nabla w_0))\cdot\nabla(u^*-w_0)\geq0\,\,\,\textrm{ for }w_0\in\mathcal{D}(\Omega).
\end{equation}
By density, \eqref{3.19} is true also for any $w_0$ in $W_0^{1,\alpha}(\Omega)$. Consider $w_0=u^*+t\varphi$, with $t\geq0$ and $\varphi\in W_0^{1,\alpha}(\Omega)$. Letting $t\rightarrow0$ and using Minty's trick as in \cite[page 94]{BM91} (see also \cite{KS00}), we conclude
\begin{equation}\label{3.20}
\sigma^*=a_0(\nabla u^*).
\end{equation}
\smallskip
\noindent\textbf{Step 4} (\textit{Lower semicontinuity of the energy}). From \eqref{3.11} and \eqref{3.13} one has
\begin{eqnarray*}
\int_\Omega\varphi\sigma_\varepsilon\cdot\nabla u_\varepsilon&\ge& \displaystyle{\int_\Omega}\varphi\sigma_\varepsilon\cdot\nabla w_\varepsilon+\displaystyle{\int_\Omega}\varphi a_\varepsilon(x,\nabla w_\varepsilon)\cdot\nabla(u_\varepsilon-w_\varepsilon)\nonumber\\
&=&-I_3-I_4.
\end{eqnarray*}
From \eqref{3.17}, \eqref{3.18} and \eqref{3.20} for any $w_0\in\mathcal{D}(\Omega)$ we have
\begin{eqnarray}\label{3.21}
&&\liminf_{\varepsilon\rightarrow0}\int_\Omega\varphi\sigma_\varepsilon\cdot\nabla u_\varepsilon\\
&\geq& \int_\Omega\varphi a_0(\nabla u^*)\cdot\nabla w_0
+\int_\Omega\varphi a_0(\nabla w_0)\cdot\nabla(u^*-w_0)\nonumber.
\end{eqnarray}
Letting $w_0$ go to $u^*$ in $W_0^{1,\alpha}(\Omega)$, one gets from \eqref{3.21}
\begin{equation}\label{3.22}
\liminf_{\varepsilon\rightarrow0}\int_\Omega\varphi\sigma_\varepsilon\cdot\nabla u_\varepsilon\geq\int\varphi a_0(\nabla u^*)\cdot\nabla u^*,
\end{equation}
$\forall\varphi\in\mathcal{D}(\Omega)$ such that $0\leq\varphi\leq1$.

\bigskip

\noindent\textbf{Step 5}. Finally, we claim that $u^*$ is the unique solution $u_0$ of \eqref{1.6}.

Let $v_0\in K_0$ and since $K_\varepsilon\rightarrow K_0$ in the Mosco sense, then there is a sequence $\bar{v}_\varepsilon\in K_\varepsilon$ such that $\bar{v}_\varepsilon\rightarrow v_0$ in $W_0^{1,\alpha}(\Omega)$. Using $\bar{v}_\varepsilon$ as a test function in \eqref{1.5} for $\varphi\in\mathcal{D}(\Omega)$, $0\leq\varphi\leq1$, one gets
\begin{equation}\label{3.23}
\int_\Omega\sigma_\varepsilon\cdot\nabla\bar{v}_\varepsilon-\int_\Omega f(\bar{v}_\varepsilon-u_\varepsilon)\geq\int_\Omega\sigma_\varepsilon\cdot\nabla u_\varepsilon\geq\int_\Omega\varphi(\sigma_\varepsilon\cdot\nabla u_\varepsilon).
\end{equation}
Recalling \eqref{3.22} and passing to the limit in $\varepsilon$ in \eqref{3.23}, we obtain
$$
\int_\Omega a_0(\nabla u^*)\cdot\nabla v_0-\int_\Omega f(v_0-u^*)\geq\int_\Omega\varphi a_0(\nabla u^*)\cdot\nabla u^*.
$$
Letting $\varphi\rightarrow1$ in the last inequality, one gets
$$
\int_\Omega a_0(\nabla u^*)\cdot\nabla(v_0-u^*)-\int_\Omega f(v_0-u^*)\geq0,\,\,\forall v_0\in K_0.
$$
The latter, combined with \eqref{3.6}, allow us to conclude that $u^*$ coincides with the unique solution $u_0$ of \eqref{1.6} and the whole sequence $u_\varepsilon\rightharpoonup u_0$ in $W_0^{1,\alpha}(\Omega)$.
\end{proof}
\begin{remark}\label{r3.1}
One can also show the convergence of the energies. More precisely,
\begin{equation}\label{3.24}
\int_\Omega a_\varepsilon(x,\nabla u_\varepsilon)\cdot\nabla u_\varepsilon\,dx\rightarrow\int_\Omega a_0(\nabla u_0)\cdot\nabla u_0\,dx.
\end{equation}
\end{remark}
\begin{proof}
For any $\varphi\in\mathcal{D}(\Omega)$ such that $0\leq\varphi\leq1$ from \eqref{3.14} we have
$$
\int_\Omega\varphi\sigma_\varepsilon\cdot\nabla u_\varepsilon\leq\int_\Omega\sigma_\varepsilon\cdot\nabla(\varphi\bar{u}_\varepsilon)-\int_\Omega u_\varepsilon\sigma_\varepsilon\cdot\nabla\varphi-\int_\Omega f\varphi(\bar{u}_\varepsilon-u_\varepsilon),
$$
which gives
\begin{equation}\label{3.25}
\limsup_{\varepsilon\rightarrow0}\int_\Omega\varphi\sigma_\varepsilon\cdot\nabla u_\varepsilon\leq\int_\Omega a_0(\nabla u_0)\cdot\nabla u_0.
\end{equation}
The latter, combined with \eqref{3.22}, implies
$$
\sigma_\varepsilon\cdot\nabla u_\varepsilon\rightarrow a_0(\nabla u_0)\cdot\nabla u_0\,\,\,\textrm{ in }\mathcal{D}'(\Omega).
$$
Since $K_\varepsilon\rightarrow K_0$ in the Mosco sense, then taking $v_0=u_0$ in \eqref{3.23}, we get
\begin{eqnarray*}
 \int_\Omega a_0(\nabla u_0)\cdot\nabla u_0&\geq&\limsup_{\varepsilon\rightarrow0}\int_\Omega\sigma_\varepsilon\cdot\nabla u_\varepsilon\\
 &\ge& \liminf_{\varepsilon\rightarrow0}\int_\Omega\sigma_\varepsilon\cdot\nabla u_\varepsilon\\
 &\geq&\int_\Omega\varphi a_0(\nabla u_0)\cdot\nabla u_0,
\end{eqnarray*}
and letting $\varphi\rightarrow1$, we obtain \eqref{3.24}.
\end{proof}
\begin{remark}\label{r3.2}
If in \eqref{1.5} we have $f_\varepsilon$ instead of $f$ and $f_\varepsilon\rightharpoonup f$ in $L^s(\Omega)$, then the conclusion of the Theorem \ref{t3.1} still holds.
\end{remark}
\begin{remark}\label{r3.3}
Since there are Lewy-Stampacchia inequalities also for the two obstacles problem (see \cite{RT11}), the Theorem \ref{t3.1} can be extended for two obstacles problems with similar assumptions.
\end{remark}
\section{Convergence of the coincidence sets}\label{s4}
In this section, using the Lewy-Stampacchia inequalities, we prove a stability result for the coincidence sets as it was done, for example, in Theorem 6:6.1 in \cite{R91}.
\begin{theorem}\label{t4.1}
Let the conditions of Theorem \ref{t3.1} hold. If, as $\varepsilon\rightarrow0$,
\begin{equation}\label{4.2}
u_\varepsilon-\psi_\varepsilon\rightarrow u_0-\psi_0\,\,\textrm{ in }\,\,L^1(\Omega),
\end{equation}
\begin{equation}\label{4.3}
(A_\varepsilon\psi_\varepsilon-f)^+\rightarrow(A_0\psi_0-f)^+\,\,\textrm{ in }\,\,L^1(\Omega),
\end{equation}
\begin{equation}\label{4.4}
A_\varepsilon u_\varepsilon\rightarrow A_0u_0\,\,\textrm{ in }\,\,\mathcal{D}'(\Omega),
\end{equation}
\begin{equation}\label{4.5}
\int_Sd(A_0\psi_0-f)\neq0,\,\,\forall S\subset\Omega\,\,\textrm{ such that }\,\,|S|>0,
\end{equation}
and
\begin{equation}\label{4.6}
A_0u_0-f=(A_0\psi_0-f)\chi_0\,\,\textrm{ a.e. in }\,\,\Omega,
\end{equation}
where $\chi_0$ is the characteristic function of the set $I_0:=\{u_0=\psi_0\}$, then the coincidence sets $I_\varepsilon:=\{u_\varepsilon=\psi_\varepsilon\}$ converge in measure, i.e.,
$$
\chi_\varepsilon\rightarrow\chi_0\,\,\textrm{ in }\,\,L^p(\Omega),\,\forall p\in[1,\infty),
$$
where $\displaystyle{\chi_\varepsilon}$ is the characteristic function of $I_\varepsilon$.
\end{theorem}
\begin{proof}
From the Lewy-Stampacchia inequalities we have
$$
f\leq A_\varepsilon u_\varepsilon\leq f+(A_\varepsilon\psi_\varepsilon-f)^+\,\,\textrm{ a.e. in }\,\,\Omega.
$$
Hence, there exists a function $q_\varepsilon\in L^\infty(\Omega)$, such that,
\begin{equation}\label{4.7}
A_\varepsilon u_\varepsilon-f=q_\varepsilon(A_\varepsilon\psi_\varepsilon-f)^+\,\,\textrm{ a.e. in }\,\,\Omega,
\end{equation}
and
\begin{equation}\label{4.8}
0\leq q_\varepsilon\leq\chi_\varepsilon\leq1\,\,\textrm{ a.e. in }\,\,\Omega.
\end{equation}
Then for a subsequence (still denoted by $\varepsilon$), one has
\begin{equation}\label{4.9}
q_\varepsilon\rightarrow q\,\,\textrm{ and }\,\,\chi_\varepsilon\rightharpoonup\chi_*\,\,\textrm{ in }\,\,L^\infty(\Omega)-\textrm{weak*}
\end{equation}
for functions $q$, $\chi_*\in L^\infty(\Omega)$. The inequalities \eqref{4.8} imply
\begin{equation}\label{4.10}
0\leq q\leq\chi_*\leq1\,\,\textrm{ a.e. in }\,\,\Omega.
\end{equation}
Using \eqref{4.3}, \eqref{4.4} and \eqref{4.9}, we pass to the limit, as $\varepsilon\rightarrow0$, in \eqref{4.7} and obtain
$$
A_0u_0-f=q(A_0\psi_0-f)^+\,\,\textrm{ a.e. in }\,\,\Omega.
$$
The latter, combined with \eqref{4.6} provides
\begin{equation}\label{4.11}
q(A_0\psi_0-f)^+=(A_0\psi_0-f)\chi_0\,\,\textrm{ a.e. in }\,\,\Omega.
\end{equation}
Note that in the region $\{A_0\psi_0>f\}$, \eqref{4.11} and \eqref{4.5} imply that $q=\chi_0$, while in $\{A_0\psi_0\leq f\}$, $\chi_0=0$. Therefore, $q\geq\chi_0$ a.e. in $\Omega$. Consequently, from \eqref{4.10} we get
$$
\chi_0\leq\chi_*\,\,\textrm{ a.e. in }\,\,\Omega.
$$
On the other hand, from \eqref{4.2} and \eqref{4.9} one has
$$
0=\int_\Omega\chi_\varepsilon(u_\varepsilon-\psi_\varepsilon)\rightarrow\int_\Omega\chi_*(u_0-\psi_0)=0,
$$
thus $\displaystyle{\chi_*}(u_0-\psi_0)=0$ a.e. in $\Omega$. Consequently, if $u_0>\psi_0$, then $\displaystyle{\chi_*}=0$, and since $0\leq\displaystyle{\chi_*}\leq1$, one obtains
$$
\chi_0\geq\chi_*\,\,\textrm{ a.e. in }\,\,\Omega.
$$
Therefore, $\displaystyle{\chi_0}=\displaystyle{\chi_*}$, and the whole sequence $\displaystyle{\chi_\varepsilon}$ converges to ${\chi_0}$ as $\varepsilon\rightarrow0$, first weakly, and since they are characteristic functions, also strongly in any $L^p(\Omega)$, for any $p\in[1,\infty)$.
\end{proof}
\begin{remark}\label{r4.1}
If $\psi_0=0$ and the right hand side is regular enough, the condition \eqref{4.6} holds automatically, since in this particular case one has porosity of the free boundary from \cite{CLRT14} (hence, the free boundary has Lebesgue measure zero), which provides \eqref{4.6}.
\end{remark}
\begin{remark}\label{r4.2}
The assumption \eqref{4.5} is a weaker version of the condition
$$
A_0\psi_0-f\neq0\,\,\textrm{ a.e. in }\,\,\Omega,\,\,\textrm{ when }\,\,A_0\psi_0\in L^1(\Omega).
$$
\end{remark}
\begin{theorem}\label{t4.2}
Let the conditions of Theorem \ref{t3.1} and also $s>n/2$. If $\psi_\varepsilon\rightarrow\psi_0$, uniformly, $\psi_0\big|_{\partial\Omega}<0$ and 
$$
\overline{\textrm{int}\{u_0=\psi_0\}}=\{u_0=\psi_0\}=I_0,
$$
then the coincidence sets $I_\varepsilon:=\{u_\varepsilon=\psi_\varepsilon\}$ converge in the Hausdorff distance to $I_0$.
\end{theorem}
\begin{proof}
Using \cite[Theorem 3.2]{FZ00}, we obtain the uniform H\"older continuity of solutions. The uniform H\"{o}lder continuity of the obstacles then implies, as $\varepsilon\rightarrow0$, the convergence $u_\varepsilon\rightarrow u_0$, uniformly in compact subsets of $\Omega$. This, in turn, provides the convergence of the coincidence sets in Hausdorff distance as in \cite{CR81} and \cite[Theorem 6:6.5]{R91}.
\end{proof}

\bigskip
 
\noindent{\bf Acknowledgments.} This work was partially supported by FCT-Portugal grant\- SFRH/BPD/92717/2013.


\begin{thebibliography}{99}

\bibitem{AP83} H. Attouch and C. Picard, \textit{Variational inequalities with varying obstacles. The general form of the limit problem}, J. Funct. Anal. 50 (1983), 329-386.

\bibitem{B11} L. Boccardo, \textit{Lewy-Stampacchia inequality in quasilinear unilateral problems and applications to the $G$-convergence}, Boll. Unione Mat. Ital. (9) 4 (2011), 275-282.

\bibitem{BD75} L. Boccardo and I.C. Dolcetta, \textit{$G$-convergenza e problema di Dirichlet unilaterale}, Boll. Unione Mat. Ital. 4 (1975), 115-123.

\bibitem{BM82} L. Boccardo and F. Murat, \textit{Nouveaux r\'{e}sultas de convergence dans des prob\`{e}mes unilateraux}, Research Notes in Math. 60 (1982), 64-85.

\bibitem{BM91} L. Boccardo and F. Murat, \textit{Homogenization of nonlinear unilateral problems}, Composite media and homogenization theory; an international center for theoretical physics workshop, Trieste, Italy, January 1990 (1991), 81-105.

\bibitem{CD99} D. Cioranesco and P. Donato, \textit{An introduction to homogenization}, Oxford lecture series in Mathematics and its applications 17, Oxford University Press Inc., New York, 1999.

\bibitem{CDD90} V.Chiado Piat, G. Dal Maso and A. Defranceschi, \textit{$G$-convergence of monotone operators}, Ann. Inst. Henri Poincar\'{e} 7 (1990), 123-160.

\bibitem{CR81} M. Codegone and J.F. Rodrigues, \textit{Convergence of the coincidence set in the homogenization of the obstacle problem}, Ann. Fac. Sci. Toulouse Math. (5) 3 (1981), 275-285.

\bibitem{CLRT14} S. Challal, A. Lyaghfouri, J.F. Rodrigues and R. Teymurazyan, \textit{On the regularity of the free boundary for quasilinear obstacle problems}, Interfaces Free Bound. 16 (2014), 359-394. 

\bibitem{DHHR11} L. Diening, P. Harjulehto, P. H\"{a}st\"{o} and M. Ruzicka, \textit{Lebesgue and Sobolev spaces with variable exponents}, Lecture Notes in Mathematics 2017 (2011).

\bibitem{E90} L.C. Evans, \textit{Weak convergence methods for nonlinear partial differential equations}, American Mathematical Society, 1990.

\bibitem{FZ03} X. Fan and Q. Zhang, \textit{Existence of solutions for $p(x)$-Laplacian Dirichlet problem}, Nonlinear Anal., Theory
Methods Appl. 52 (2003), 1843-1852.

\bibitem{FZ00} X. Fan and D. Zhao, \textit{The quasi-minimizer of integral functionals with $m(x)$ growth condition}, Nonlinear Anal. 39 (2000), 807-816.

\bibitem{FZ01} X. Fan and D. Zhao, \textit{On the spaces $L^{p(x)}(\Omega)$ and $W^{m,p(x)}(\Omega)$}, J. Math. Anal. Appl. 263 (2001), 424-446.

\bibitem{KS00} D. Kinderlehrer and G. Stampacchia, \textit{An introduction to variational inequalities and their applications}, SIAM Classics in Applied Mathematics 31, 2000, xxii+306p.

\bibitem{M78} F. Murat, \textit{Compacit\'{e} par compensation}, Ann. Scuola Norm. Sup. Pisa. Cl. Sci (4) 5 (1978), 489-507.

\bibitem{P97} A.A. Pankov, \textit{$G$-convergence and homogenization of nonlinear partial differential operators}, Kluwer Academic Publishers, Dordrecht, 1997.

\bibitem{R91} J.F. Rodrigues, \textit{Obstacle problems in mathematical physics}, North-Holland Mathematics Studies 134 (Notas de Matematica 114), Elsevier Science Publishers B.V., 1991.

\bibitem{RT11} J.F. Rodrigues and R. Teymurazyan, \textit{On the two obstacles problem in Orlicz-Sobolev spaces and applications}, Complex Var. Elliptic Equ. 56 (2011), 769-787. 

\bibitem{S05} G.V. Sandrakov, \textit{Homogenization of variational inequalities for obstacle problems}, Sbornik, Mathematics 196 (2005), 541-560.

\bibitem{S06} G.V. Sandrakov, \textit{Homogenization of nonlinear equations and variational inequalities with obstacles}, Doklady Mathematics 73 (2006), 178-181.

\bibitem{Z91} V.V. Zhikov, \textit{Lavrent'ev effect and the homogenization of nonlinear variational problems}, Differ. Equations 27 (1991), 32-39. 

\bibitem{ZP101} V.V. Zhikov and S.E. Pastukhova, \textit{On the compensated compactness principle}, Doklady Mathematics 82 (2010), 590-595.

\bibitem{ZP102} V.V. Zhikov and S.E. Pastukhova, \textit{Homogenization of monotone operators under conditions of coercivity and growth of variable order}, Mathematical Notes 90 (2011), 48-63.

\end{thebibliography}
\end{document}